\documentclass[12pt]{amsart}

\usepackage{amssymb, amsfonts}
\usepackage{url}
\usepackage[all,arc]{xy}
\usepackage{amscd}
\usepackage{mathrsfs}
\usepackage{geometry}
\usepackage[colorlinks,citecolor=blue]{hyperref}
\usepackage{comment}
\usepackage{enumerate}
\usepackage{graphicx}
\usepackage{bm}
\usepackage{color}
\usepackage{anyfontsize}
\usepackage{caption}

\numberwithin{equation}{section}

\theoremstyle{plain}
\newtheorem{theorem}{Theorem}[section]

\newtheorem{lemma}[theorem]{Lemma}
\newtheorem{proposition}[theorem]{Proposition}

\newtheorem{definition}[theorem]{Definition}

\theoremstyle{remark}

\begin{document}


\title[Some new functionals]{Some new functionals related to free boundary minimal submanifolds}

\author{Tianyu Ma, Vladimir Medvedev}

\address{Faculty of Mathematics, National Research University Higher School of Economics, 6 Usacheva Street, Moscow, 119048, Russian Federation}

\email{tma@hse.ru, vomedvedev@hse.ru}



\begin{abstract}
The metrics induced on free boundary minimal surfaces in geodesic balls in the upper unit hemisphere and hyperbolic space can be characterized as critical metrics for the functionals $\Theta_{r,i}$ and $\Omega_{r,i}$, introduced recently by Lima, Menezes and the second author. In this paper, we generalize this characterization to free boundary minimal submanifolds of higher dimension in the same spaces. We also introduce some functionals of the form different from $\Theta_{r,i}$, and show that the critical metrics for them are the metrics induced by free boundary minimal immersions into a geodesic ball in the upper unit hemisphere. In the case of surfaces, these functionals are bounded from above and not bounded from below. Moreover, the canonical metric on a geodesic disk in a 3-ball in the upper unit hemisphere is maximal for this functional on the set of all Riemannian metric of the topological disk.
 \end{abstract}

\maketitle


\newcommand\cont{\operatorname{cont}}
\newcommand\diff{\operatorname{diff}}

\newcommand{\dvol}{\text{dA}}
\newcommand{\Ric}{\operatorname{Ric}}
\newcommand{\Hess}{\operatorname{Hess}}
\newcommand{\GL}{\operatorname{GL}}
\newcommand{\myO}{\operatorname{O}}
\newcommand{\myP}{\operatorname{P}}
\newcommand{\eye}{\operatorname{Id}}
\newcommand{\myF}{\operatorname{F}}
\newcommand{\Vol}{\operatorname{Vol}}
\newcommand{\odd}{\operatorname{odd}}
\newcommand{\even}{\operatorname{even}}
\newcommand{\ol}{\overline}
\newcommand{\mye}{\operatorname{E}}
\newcommand{\myo}{\operatorname{o}}
\newcommand{\myt}{\operatorname{t}}
\newcommand{\irr}{\operatorname{Irr}}
\newcommand{\mydiv}{\operatorname{div}}
\newcommand{\curl}{\operatorname{curl}}
\newcommand{\re}{\operatorname{Re}}
\newcommand{\im}{\operatorname{Im}}
\newcommand{\can}{\operatorname{can}}
\newcommand{\scal}{\operatorname{scal}}
\newcommand{\tr}{\operatorname{trace}}
\newcommand{\sgn}{\operatorname{sgn}}
\newcommand{\SL}{\operatorname{SL}}
\newcommand{\myspan}{\operatorname{span}}
\newcommand{\mydet}{\operatorname{det}}
\newcommand{\SO}{\operatorname{SO}}
\newcommand{\SU}{\operatorname{SU}}
\newcommand{\specl}{\operatorname{spec_{\mathcal{L}}}}
\newcommand{\fix}{\operatorname{Fix}}
\newcommand{\id}{\operatorname{id}}
\newcommand{\grad}{\operatorname{grad}}
\newcommand{\singsup}{\operatorname{singsupp}}
\newcommand{\wave}{\operatorname{wave}}
\newcommand{\ind}{\operatorname{ind}}
\newcommand{\mynull}{\operatorname{null}}
\newcommand{\inj}{\operatorname{inj}}
\newcommand{\arcsinh}{\operatorname{arcsinh}}
\newcommand{\Spec}{\operatorname{Spec}}
\newcommand{\Ind}{\operatorname{Ind}}
\newcommand{\Nul}{\operatorname{Nul}}
\newcommand{\inrad}{\operatorname{inrad}}
\newcommand{\mult}{\operatorname{mult}}
\newcommand{\Length}{\operatorname{Length}}
\newcommand{\Area}{\operatorname{Area}}
\newcommand{\Ker}{\operatorname{Ker}}
\newcommand{\floor}[1]{\left \lfloor #1  \right \rfloor}

\newcommand\restr[2]{{
  \left.\kern-\nulldelimiterspace 
  #1 
  \vphantom{\big|} 
  \right|_{#2} 
  }}


\section{Introduction}

Minimal submanifolds are critical points of the volume functional among all compactly supported variations. More precisely, let $\Sigma^k$ be a $k$-dimensional submanifold of some Riemannian ambient space $\overline{M}$. We say that $\Sigma^k$ is minimal if for every $p\in \Sigma$, there is a neighbourhood $U_p$ of $p$ in $\Sigma$ such that $U_p$ minimizes the $k$-dimensional volume among all of its smooth variations with fixed boundary $\partial U_p$. This is equivalent to the condition of vanishing mean curvature. Let $\Sigma^k$ be a compact manifold with boundary with the unit outer normal vector field $\eta$ on $\partial\Sigma^k$. We say that an isometric immersion $\Phi: \Sigma^k\rightarrow \overline{M}$ is a \textit{free boundary minimal immersion} (FBMI for short) if it is a critical point of the volume functional,  $\Phi(\partial\Sigma^k)\subset\partial\overline{M}$, and the image of $\eta$ is perpendicular to $\partial\overline{M}$.\\

In the case where the ambient space $\overline{M}$ is the closed unit ball in the Euclidean space $\mathbb{E}^{m+1}$, FBMI's have a decent description. Namely, all coordinate functions $\phi_i=x_i\circ\Phi$, $i=0,\dots,m$, are Steklov eigenfunctions with the eigenvalue $\sigma=1$, i.e., they are harmonic functions on $\Sigma^k$, and $\dfrac{\partial \phi_i}{\partial \eta}=\phi_i$ on $\partial\Sigma^k$. There is a similar description of FBMI's in the case that the ambient space $(\overline{M},\overline g)$ is a closed ball in a space of constant curvature, which is also the focus of this paper.\\

We consider the standard upper unit hemisphere $\mathbb{S}^m_+$ in $\mathbb{E}^{m+1}$ and also the Lobachevsky $m$-space $\mathbb{H}^m$ as the hyperboloid in the Minkowski space $\mathbb{M}^{1,m}$. When $(\overline{M},\overline g)$ is a suitably placed geodesic ball in $\mathbb{S}^m_+$ or in $\mathbb{H}^m$,  a map $\Phi:(\Sigma^k,g)\rightarrow M$ is an FBMI if and only if the coordinate functions $\phi_i$ are solutions to the \textit{Robin problem}. Namely, they satisfy the following system of equations
\begin{align}
\label{eqn:robin general}
 \begin{cases}
 \Delta_g \phi_i= \lambda \phi_i&\text{ in}\, \Sigma,\\
 \dfrac{\partial \phi_i}{\partial \eta}=\sigma \phi_i &\text{ on} \,\partial\Sigma,
 \end{cases}
 \end{align}
 for suitably prescribed values of $\lambda$ and $\sigma$ (see below). In this paper, we take the sign convention of the Laplace operator to be $\Delta_g=-\mydiv_g\circ \nabla^g$. The existence of FBMI into geodesic balls in $\mathbb{S}^m_+$ or in $\mathbb{H}^m$ corresponds to the case where solutions of a system of the type \eqref{eqn:robin general} exist and form an isometric immersion into the respective ambient space; see ~\cite{lima2023eigenvalue,medvedev2025free}. More generally, when the map $\Phi$ is just an immersion (not necessarily isometric), the coordinate functions are solutions of this system if and only if $\Phi$ is harmonic. In this case, $\Phi$ is a \textit{free boundary harmonic map} (FBHM for short). 
  
In order to find the metrics induced by FBMI's in geodesic balls in $\mathbb{S}^m_+$ and $\mathbb{H}^m$, we characterize them as critical points of certain functionals. Let us recall the definition (see~\cite{nadirashvili1996berger,el2000riemannian,karpukhin2022laplace}).
\begin{definition}
Let $g$ be a metric on a manifold $\Sigma^k$, and let $F$ be a functional defined on some set $\mathcal{S}(\Sigma^k)$ of metrics on $\Sigma^k$ containing $g$. We say that $g$ is an extremal metric for $F$ if for all one-parameter smooth family metrics $g(t)$ in $\mathcal{S}(\Sigma^k)$  with $g(0)=g$, we have 
\begin{align*}
F(g(t))\leqslant F(g)+o(t)\ \mathrm{or}\ F(g(t))\geqslant F(g)+o(t).
\end{align*}
In particular, if both left and right derivatives of $g(t)$ exist at $t=0$, the above is equivalent to
\begin{align*}
\lim_{t\rightarrow 0^+}\dfrac{d F(g(t))}{dt}\times \lim_{t\rightarrow 0^-}\dfrac{d F(g(t))}{dt}\leqslant 0.
\end{align*}
\end{definition}
Hence, minimal and maximal metrics are special cases of extremal metrics.

To the best of our knowledge, the description of metrics on minimal submanifolds as extremal metrics for certain functionals on a subset of Riemannian metrics goes back to the seminal paper by Nadirashvili~\cite{nadirashvili1996berger}, where he considered the first normalized eigenvalue functional $\Lambda_1$ on closed surfaces. In this case, the unit Euclidean sphere of certain dimension plays the role of the ambient manifold $\overline M$. Later, this approach was generalized in the papers~\cite{el2000riemannian,el2008laplacian} to the case of the $k$-th normalized eigenvalue functional $\Lambda_k$ on closed manifolds of arbitrary dimension. Notice that in the two-dimensional case, the functional $\Lambda_k$ is bounded from above (see~\cite{korevaar1993upper,hassannezhad2011conformal} for some implicit upper bounds and~\cite{yang1980eigenvalues,li1982new,el1983volume,kokarev2014variational,karpukhin2016upper} for some explicit upper bounds). It turns out that in the higher dimensional case, the functional $\Lambda_k$ is not bounded from above ~\cite{colbois1994riemannian}, but it is still bounded from above in any conformal class (see e.g.,~\cite{korevaar1993upper,hassannezhad2011conformal} again). Further, relying on an analogy between the spectrum of a closed Riemannian manifold and the Steklov spectrum of a Riemannian manifold with boundary, Fraser and Schoen in~\cite{fraser2016sharp,fraser2013minimal} characterized metrics of free boundary minimal surfaces in the unit Euclidean ball as extremal metrics for the $k$-th normalized Steklov eigenvalue. Its higher dimensional generalization was considered in the papers~\cite{fraser2019shape,karpukhin2022laplace}. This functional is bounded from above for compact surfaces with boundaries (see~\cite{hassannezhad2011conformal,colbois2011isoperimetric} for some implicit upper bounds and~\cite{fraser2011first,kokarev2014variational,girouard2012upper,karpukhin2015bounds,medvedev2022degenerating} for some explicit upper bounds), but is not bounded even in a conformal class, when a compact manifold with boundary has dimension at least 3 (see~\cite{colbois2019compact}). For more information about these functionals, we refer an interested reader to the surveys~\cite{penskoi2013extremal,girouard2017spectral,colbois2024some}. We also mention the papers~\cite{petrides2023extremal,petrides2024shape,petrides2022variational,petrides2023non,petrides2023laplace,petrides2024critical}, where extremal metrics for functionals containing combinantions of Laplace and Steklov eigenvalues were studied, as well as their connection to the theory of (free boundary) minimal submanifolds. Very recently, Lima and Menezes introduced in~\cite{lima2023eigenvalue} the functional $\Theta_{r}$ on compact surfaces with boundary and proved that this functional is bounded from above. They characterized maximal metrics for this functional as metrics of free boundary minimal surfaces in a spherical cap of radius $r\in (0,\pi/2)$. These results were generalized by the second author in~\cite{medvedev2025free}: the functionals $\Theta_{r,k}$ and $\Omega_{r,k}$ on compact surfaces with boundaries were defined for arbitrary $k\in \mathbb N^+$, and it was proved that extremal metrics for these functional are exactly the metrics induced by FBMI into a spherical cap or a geodesic ball of radius $r$ in the hyperbolic space, respectively. It was also shown that $\Theta_{r,k}$ is bounded from above but not bounded from below on Riemannian metrics on a given surface with boundary, and $\Omega_{r,k}$ is bounded from below but not bounded from above on the same set. 

In this paper, we consider the following functionals on Riemannian metrics on a given compact manifold with boundary $\Sigma^k$
$$
\Theta_{r,i}(\Sigma,g) := (\theta_0(g)\cos^2 r+\theta_i(g)\sin^2 r)|\partial \Sigma|_g + 2\vert \Sigma\vert_g,
$$
$$
\Omega_{r,i}(\Sigma,g):=(-\omega_0(g) \cosh^2 r+\omega_i(g) \sinh^2 r)|\partial\Sigma|_g + 2\vert \Sigma\vert_g,
$$
where the functionals  $\Omega_{r,i}(\Sigma,\cdot)$ and $\Theta_{r,i}(\Sigma,\cdot)$ are defined on the set $\mathcal{R}(\Sigma^k)$ of all Riemannian metrics on $\Sigma^k$ and its subclass  $\mathcal{R}_{k}(\Sigma^k)$ (see Section \ref{sec:2.1} for the definition), respectively. The term $\theta_i=\sigma_i(g,-k)$ with $i\geqslant 1$ is the $i$-th element of the sequence of eigenvalues $\{\sigma_i(g,-k)\}$ arranged in increasing order, when we fix the Laplacian eigenvalue $\lambda=k$ in \eqref{eqn:robin general} (we call these eigenvalues the \textit{Steklov eigenvalues of frequency $k$}). Also $\omega_i=\sigma_i(g,k)$, the \textit{Steklov eigenvalues of frequency $-k$}, are defined in a similar way. They will be our first class of candidates that characterize metrics admitting the desired FBMI or FBHM. These functionals $\Theta_{r,i}$ and $\Omega_{r,i}$ have the exact same form as in the 2-dimensional case (see~\cite{lima2023eigenvalue,medvedev2025free}). A reasoning of why the functionals are in such forms is given in Section~\ref{subsec:coeff}.

We also introduce the following functionals defined on the set $\mathcal{R}(\Sigma^k)$ using Laplacian eigenvalues 
$$
 \Xi_{r,i}^+(\Sigma,g):=\min\left\lbrace \lambda_0(g,-\tan(r)), \lambda_i(g,  \cot(r))\right\rbrace|\Sigma|_g^{2/k},\ i\geqslant 1, \, 0<r<\frac{\pi}{2},
$$
 and 
$$
 \Xi_{r,i}^-(\Sigma,g):=\min\{\lambda_0(g, \tanh(r)),\lambda_i(g, \coth(r))\}|\Sigma|_g^{2/k},\, i\geqslant 1, \, r>0.
$$
We will demonstrate the relationship between the critical points of these functionals and metrics admitting FBMI or FBHM into geodesic balls in $\mathbb{S}^m_+$ or $\mathbb{H}^m$, summarized in the following theorems. 

\begin{theorem}\label{2.2}
Let $\Sigma^k$ be a $k$-dimensional compact smooth manifold with boundary. For $r<\dfrac{\pi}{2}$ and $i\geqslant 1$. Let $\mathbb B^m(r)$ denote a ball of radius $r$ in $\mathbb{S}^m$ centered at the point $(1,0,\ldots,0)$ and $V_i(g)$ the eigenspace of the Steklov eigenvalue $\theta_i$ with frequency $k$. Then the following is true:
\begin{itemize} 
\item If $g$ is an extremal metric for $\Theta_{r,i}$ on the class $\mathcal{R}_{k}(\Sigma)$, then there exist Steklov eigenfunctions with frequency $k$: $v_0\in V_0(g)$ and $v_1,\ldots,v_m\in V_i(g)$ such that the following map
$$\Phi:\Sigma\rightarrow \mathbb{R}^{m+1},\ x\mapsto (v_0(x),v_1(x),\ldots,v_m(x))$$
is in fact an FBMI into $\mathbb B^m(r)$.
\item If $g$ is an extremal metric for $\Theta_{r,i}$ on a Riemannian conformal class $[g]\cap\mathcal{R}_{k}(\Sigma)$, then there exist Steklov eigenfunctions of frequency $k$ $v_0\in V_0(g)$ and $v_1,\ldots,v_m\in V_i(g)$ such that the map $\Phi$ is an FBHM into $\mathbb B^m(r)$.
\end{itemize}
\end{theorem}

\begin{theorem}\label{2.3}
Given $i\geqslant 1$, suppose for some $g\in \mathcal{R}_{k}(\Sigma^k)$ we have either $\theta_i(g)<\theta_{i+1}(g)$ or $\theta_i(g)>\theta_{i-1}(g)$. Let $V_i(g)$ be the eigenspace of the Steklov eigenvalue $\theta_i$ of frequency $k$.
\begin{itemize}
\item Suppose there exists $v_0\in V_0(g)$ and independent $v_1,\ldots,v_m\in V_i(g)$ such that
\begin{enumerate}
\item $dv_0\otimes dv_0+\displaystyle\sum_{j=1}^m dv_j\otimes dv_j=g,$
\item $v_0^2+\displaystyle\sum_{j=1}^m v_j^2=1.$
\end{enumerate}
Then there is some $0<r<\dfrac{\pi}{2}$ such that $g$ is an extremal metric for $\Theta_{r,i}$ in the set $\mathcal{R}_{k}(\Sigma)$.
\item If there exist $v_0\in V_0(g)$ and independent $v_1,\ldots,v_m\in V_i(g)$ such that $v_0^2+\displaystyle\sum_{j=1}^m v_j^2=1$, the metric $g$ is extremal for $\Theta_{r,i}$ in the set $[g]\cap\mathcal{R}_{k}(\Sigma)$ for some $0<r<\dfrac{\pi}{2}$.
\end{itemize}
\end{theorem}

In these two theorems, $\mathcal{R}_{k}(\Sigma)$ denotes the subset of Riemannian metrics on $\Sigma$, which do not admit $k$ as a Dirichlet eigenvalue. This restriction is necessary (see Section~\ref{sphere}) in the spherical case. 

In Section~\ref{sec:hyper}, we also prove similar theorems for FBMI and FBHM into a geodesic ball in $\mathbb H^m$.

It turns out that the metrics induced by FBMI into geodesic balls in the hemisphere can be obtained in a different approach. Namely, instead of the functional $\Theta_{r,i}$ on $\mathcal R_{k}(\Sigma^k)$, one can take the functional $\Xi_{r,i}^+(\Sigma,g)$, defined above, on the set of \textit{all metrics} on $\Sigma$. This functional admits extremal metrics, which are metrics induced by FBMI into geodesic balls in the unit hemisphere. The formal statement is as follows.

\begin{theorem}\label{4.4}
The functional $\Xi_{r,i}^-(\Sigma,g)$ has no extremal metrics on $\mathcal{R}(\Sigma)$, and the functional $\Xi_{r,i}^+(\Sigma,g)$ has extremal metrics on $\mathcal{R}(\Sigma)$ only if
\begin{align}
\label{eqn:laplacian eigenvalue equal}
\lambda_0(g,- \tan(r))=\lambda_i(g,  \cot(r)).
\end{align}
Moreover, if $g$ is an extremal metric for $\Xi_{r,i}^+(\Sigma,g)$, then there exist $v_0\in E_0(g,-\tan r)$ and $v_1,\ldots,v_m\in E_i(g,\cot  
r)$ such that the map $(v_0,v_1,\ldots,v_m)$ induces an FBMI from $\Sigma$ into the geodesic ball $\mathbb B^m(r)\subset \mathbb{S}^m_+$ centred at $(1,0,\ldots,0)$.
\end{theorem}

Here $E_i(g,\sigma)$ is the $i$-th eigenspace of problem~\ref{eqn:robin general} with $\sigma$ given. Theorem~\ref{4.4} actually states that it is senseless to consider the functional  $\Xi_{r,i}^-(\Sigma,g)$ from the point of view of extremal metrics. However, considering the functional $\Xi_{r,i}^+(\Sigma,g)$ is meaningful. Particularly, one can be interested in finding maximal and minimal metrics for $\Xi_{r,i}^+(\Sigma,g)$. This leads us to the question of whether this functional is bounded. In Proposition~\ref{prop:bound} we show that both functionals $\Xi_{r,i}^+(\Sigma,g)$ and $\Xi_{r,i}^-(\Sigma,g)$ are bounded from above on the set of Riemannian metrics for any compact surface with boundary. For $\Xi_{r,1}^+(\Sigma,g)$, we were also able to find an explicit upper bound in the following theorem, which was inspired by Theorem A in~\cite{lima2023eigenvalue}
 \begin{theorem}\label{thm:bound}
 For any compact surface $\Sigma$ with boundary, one has
 $$
  \Xi_{r,1}^+(\Sigma,g)\leqslant 4\pi(1-\cos r)(\gamma+l),
 $$
 where $\gamma$ and $l$ are, respectively, the genus and the number of boundary components. If $\Sigma$ is a topological disk, then  
 $$
  \Xi_{r,1}^+(\Sigma,g)\leqslant 4\pi(1-\cos r),
  $$
  and the equality holds if and only if $\Sigma$ is a geodesic disk in $\mathbb B^2(r)\subset \mathbb S^2_+$.
 \end{theorem}
 However, in Proposition~\ref{prop:unbound}, we show that both functionals are not bounded from below even in any conformal class. These results imply that one can find maximal metrics for $\Xi_{r,i}^+(\Sigma,g)$ on $\mathcal R(\Sigma)$, but there are no minimal metrics for it on $\mathcal R(\Sigma)$. Moreover, the canonical metric on a geodesic disk in $\mathbb B^2(r)\subset \mathbb S^2_+$ is maximal for $\Xi_{r,1}^+(\Sigma,g)$ when $\Sigma$ is a topological disk. We conjecture, that the metric on the critical spherical catenoid in a geodesic ball in $\mathbb S^3_+$ (see~\cite{do1983rotation}) is maximal for $\Xi_{r,1}^+(\Sigma,g)$ on the topological annulus. There is also a candidate for a maximal metric for $\Xi_{r,1}^+(\Sigma,g)$ on the topological M\"obius band.  It would be interesting to find other maximal metrics. 
 
 In conclusion, we would like to emphasize that the main advantage of the functional $\Xi_{r,i}^+$ with respect to $\Theta_{r,i}$ is that we do not restrict the set of all Riemannian metrics on $\Sigma$ to the set of admissible metrics. Particularly, in Theorem~\ref{thm:bound}, we find a maximal metric for a topological disk on the set of all Riemannian metrics.

\subsection{Paper organization} This paper is organized as follows. Section~\ref{sphere} is dedicated to FBMI in geodesic balls in $\mathbb S^n_+$. More precisely, in Section \ref{sec:2.1}, we prove Theorem~\ref{2.2}. Here we also show Theorem~\ref{2.3}. In Section \ref{subsec:coeff}, we give a justification why the functional $\Theta_{r,i}$ is in the same form as the two-dimensional case given by \cite{lima2023eigenvalue}. In Section~\ref{sec:hyper} we prove similar theorems for the case of FBMI in geodesic balls in $\mathbb H^n$. In Section \ref{sec:alt functional}, we treat the functionals $\Xi^\pm_{r,i}$: Here we prove Theorems~\ref{4.4} and~\ref{thm:bound}.

\subsection{Acknowledgments} The authors would like to  thank Iosif Polterovich and Mikhail Kerpukhin for their comments on the preliminary version of the manuscript. The work of the second author is an output of a research project implemented as part of the Basic Research Program at the National Research University Higher School of Economics (HSE University).

\section{FBMI in spheres}\label{sphere}
\subsection{Critical metrics of sphere FBMI functional}
\label{sec:2.1}
We follow the notation conventions in \cite{lima2023eigenvalue,medvedev2025free}. Let $\Sigma^k$ be a $k$-dimensional compact manifold with (sufficiently smooth) boundary. Recall that $\mathcal{R}_{k}(\Sigma)$ denotes the subset of Riemannian metrics on $\Sigma$ such that the following PDE system has no non-zero solutions
\begin{equation*}\label{sys:dsphere}
\left\{ \begin{aligned} 
& -ku+ \Delta_g u =0\ &\mathrm{in}\ \Sigma,\\
&  u =0\ &\mathrm{on}\ \partial\Sigma.
\end{aligned} \right.
\end{equation*}
In other words, $\mathcal{R}_{k}(\Sigma)$ denotes the set of metrics on $\Sigma$, for which $k$ is not a Dirichlet eigenvalue. For $g\in\mathcal{R}_{k}(\Sigma)$, we consider the following linear PDE system with parameter $\theta$ 
\begin{equation}
\left\{ \begin{aligned} 
-ku+\Delta_g u=0\ &\mathrm{in}\ \Sigma,\\
\dfrac{\partial u}{\partial \eta}=\theta u\ &\mathrm{on}\ \partial\Sigma,
\end{aligned} \right.
\end{equation}
where $\eta$ in the unit outer normal to $\partial\Sigma$ in $(\Sigma,g)$. The Steklov eigenvalues of $g$ with frequency $k$ are real numbers $\theta$, for which the above system has non-trivial solutions. Arrange these eigenvalues in increasing order, counted with multiplicity, by $\theta_i$ for $i\geqslant 0$. Let $V_i(g)$ be the space of eigenfunctions corresponding to $\theta_i$. It is known that $( \theta_i)_{i\in\mathbb N}$ is discrete with $\theta_i\rightarrow +\infty$, as $i\to\infty$ and $\theta_0$ is simple, when $\Sigma$ is connected.\\
\\
For any $0<r<\dfrac{\pi}{2}$ and $i\geqslant 1$, let $\Theta_{r,i}$ be the following functional on $\mathcal{R}_{k}(\Sigma)$ as in \cite{medvedev2025free}
\begin{align}
\label{eqn:sphere functional}
\Theta_{r,i}(\Sigma,g) := (\theta_0(g)\cos^2 r+\theta_i(g)\sin^2 r)\vert\partial \Sigma\vert_g + 2\vert \Sigma\vert_g.
\end{align} 
This is exactly the same functional as in \cite{lima2023eigenvalue}, for which the critical metrics induce FBMI into unit spherical caps by Laplace eigenfunctions when $k=2$. We will see that this functional also serves the same purpose, when $k>2$.\\

The first problem arising in finding extremal metrics of this functional is that the eigenvalues $\theta_i(g_t)$ can be non-differentiable even for a smooth family of metrics $(g_t)_{t\in[-\varepsilon,\varepsilon]}$. Fix any $\varepsilon>0$ and let $(g_t)_{ [-\varepsilon,\varepsilon]}\in \mathcal{R}_{k}(\Sigma)$ be a smooth family of admissble metrics on $\Sigma$. Then each function $t\mapsto\theta_i(g_t)$ is Lipschitz in $t$ (see e.g. \cite[Claim 4]{medvedev2025free}). Thus, the functional $\Theta_{r,i}(g(t))$ is differentiable almost everywhere. We compute an explicit form of the derivative of $\Theta_{r,i}(g_t)$, which is well-defined for almost every $t$, following the approach in \cite{lima2023eigenvalue}. Let $g^b_t$ be the restriction of $g_t$ to $\partial \Sigma$, and denote $dA_t$ and $dL_t$ the volume density, induced by $g_t$ and $g^b_t$ on $\Sigma$ and $\partial\Sigma$, respectively. When $\theta_i(g_t)$ is differentiable, its derivative is given by
\begin{align*}
\dfrac{d}{dt}\left(\theta_i(g_t) \right)= & \int_{\Sigma} \left\langle \dfrac{1}{2}\left(-k u_i^2(t)+ \vert \nabla^t u_i(t) \vert_{g_t}^2 \right)g_t- du_i(t)\otimes du_i(t), \dot{g}(t) \right\rangle_{g_t} dA_t\\
& -\theta_i(g_t)\int_{\partial \Sigma} \dfrac{1}{2} u_i^2(t) \left\langle g^b_t,\dot{g}^b_t \right\rangle_t dL_t,
\end{align*}
where $u_i(t)$ is some eigenfunction of $\theta_i(g_t)$ such that $\Vert u_i(t)\Vert_{L^2(\partial \Sigma, dL_t)}=1$.
Therefore, the weak derivative of $\Theta_{r,i}(g_t)$ is given  by
\begin{equation}
\label{sphere functioanl dev}
\begin{aligned}
 Q_r^s(u_0(t),u_i(t)):= & \dfrac{d}{dt}\left( \Theta_{r,i}(g_t) \right)\\
 =& \int_{\Sigma} \left\langle \tau^k_t(\cos  r \cdot u_0(t))+\tau^k_t(\sin r \cdot u_i(t))+ g,\dot{g}_t \right\rangle_t dA_t \\
& +\int_{\partial\Sigma} F^k_t(\cos r \cdot u_0(t),\sin r\cdot  u_i(t))\left\langle g^b, \dot{g}^b \right\rangle_t dL_t.
\end{aligned}
\end{equation}
Here the functions $\tau^k_t$ and $F^k$ are defined analogously as in \cite{lima2023eigenvalue} by
\begin{align*}
 \tau^k(u) =&\vert \partial\Sigma_t\vert\left( \dfrac{1}{2}\left(-ku^2+\vert \nabla^t u\vert_{g_t}^2\right)g-du\otimes du  \right),\\
F^k(v_1,v_2) =&\dfrac{1}{2} \biggr( \theta_0\left(\cos^2 r -\vert \partial\Sigma_{g_t}\vert v_1^2 \right)+ \theta_i\left(\sin^2 r -\vert \partial\Sigma_{g_t}\vert v_2^2 \right)\biggr ).
\end{align*}

We want to recall several facts from \cite{lima2023eigenvalue}. The weak derivative $\dfrac{d}{dt}\left( \Theta_{r,i}(g_t) \right)$ given by \eqref{sphere functioanl dev} is a sum of two integrals. We call the integral on $\Sigma$ the interior component and the integral on $\partial\Sigma$ the boundary component. We can rewrite the formula \eqref{sphere functioanl dev} as the following inner product on the space $L^2(S^2(\Sigma,g))\times L^2(\partial\Sigma,g^b)$
\begin{align*}
\biggr\langle  \left(\dot{g}_t,\langle g^b,\dot{g}^b \rangle_t  \right) , \left(  \tau^k_t(\cos  r \cdot u_0(t))+\tau^k_t(\sin r \cdot u_i(t))+ g,   F^k_t(\cos r \cdot u_0(t),\sin r\cdot  u_i(t))                  \right)   \biggr\rangle;
\end{align*}
see \cite[section 5]{lima2023eigenvalue} for details. Also, note that $\theta_0(t)$ is a simple eigenvalue for connected $\Sigma^k$, so we can assume that $u_0(t)$ has a unique choice for each $t$. Thus, for a fixed smooth family of metrics $(g_t)_{t\in[-\varepsilon,\varepsilon]}$, for each given $t$, the above inner product is actually a function of $u_i^2(t)$. 

Although $\Theta_{r,i}(g_t)$ is not always differentiable, for a smooth family of metrics $(g_t)_{t\in[-\varepsilon,\varepsilon]}$, the functional $Q^s_r(u_0(t),u_i(t))$ given by the formula of \eqref{sphere functioanl dev} is always defined. Now suppose $g$ is an extremal metric for $\Theta_{r,i}$. Let $(g_t)_{t\in[-\varepsilon,\varepsilon]}$ be a smooth family such that $\Theta_{r,i}(g_t)\leqslant \Theta_{r,i}(g)+o(t)$ with $g_0=g$. Since $\Theta_{r,i}(g_t)$ is absolute continuous, by the Lebesgue fundamental theorem of calculus, we see there exist $t_j^+\rightarrow 0^+$ and $t_j^-\rightarrow 0^-$ such that the $L^2(\partial\Sigma,g^b(t_j^{\pm}))$-normalized eigenfunctions $u_i(t_j^{\pm})$ of $\theta_i(t_j^{\pm})$ satisfy
\begin{align*}
\limsup_{j\rightarrow\infty}  Q^s_r(u_0(t_j^+),u_i(t_j^+))\leqslant 0,\ \liminf_{j\rightarrow\infty}  Q^s_r(u_0(t_j^-),u_i(t_j^-))\geqslant 0.
\end{align*}
An analogous statement also holds for the case, when  $\Theta_{r,i}(g_t)\geqslant \Theta_{r,i}(g)+o(t)$. Therefore, we can argue as in \cite[lemma 3]{lima2023eigenvalue} to prove exactly the same statement as follows
\begin{lemma}
\label{lemma:sphere vanish derivative}
Suppose $g$ is an extremal metric of $\Theta_{r,i}$. For any $(h,\psi)\in L^2(S^2(\Sigma,g))\times L^2(\partial\Sigma,g^b)$, there exists $u_i$, which is an eigenfunction corresponding to $\theta_i$ with $\Vert u_i \Vert_{L^2(g^b)}  =1$ such that
\begin{align*}
\biggr\langle  \left( h,\psi  \right) , \left(  \tau^k_t(\cos  r \cdot u_0(t))+\tau^k_t(\sin r \cdot u_i(t))+ g,   F^k_t(\cos r \cdot u_0(t),\sin r\cdot  u_i(t))                  \right)   \biggr\rangle=0.
\end{align*}
\end{lemma}

Now we can prove Theorem \ref{2.2}, which extends the results by Lima and Menezes in \cite{lima2023eigenvalue} and the second author in~\cite{medvedev2025free} to the higher dimensional case.

\begin{proof}[Proof of Theorem \ref{2.2}]
We begin by proving the first statement of the theorem.
For fixed $l\geqslant 1$, denote $V_l(g)$ the eigenspace corresponding to the eigenvalue $\theta_l$ for $g\in \mathcal{R}_{k}(\Sigma)$ and $l=0\ \mathrm{or} \ i$. For any $g\in\mathcal{R}_{k}(\Sigma)$, we define the following subset of $ L^2(S^2(\Sigma),g)\times L^2(\partial\Sigma,g^b)$
\begin{align*}
\mathcal{C}_r:=\biggr\lbrace & \left(\tau^k (\cos r u_0)+\tau^k(\sin r u_i)+  g,\ F^k(\cos r \cdot u_0,\sin  r \cdot u_i) \right): \\
 &  \ u_i\in V_i(g),\ L^2(u_l,g^b)=1 \biggr\rbrace.
\end{align*}
This subset is actually obtained by the formula of the weak derivative of $\Theta_{r,i}(g_t)$ with $g(0)=g$. If $g\in\mathcal{R}_{k}(\Sigma)$ is a critical point of the functional $\Theta_{r,i}$, the point $(0,0) $ is in the convex hull of the set $\mathcal{C}_r(g)$ by the geometric Hahn-Banach theorem (see \cite[pp. 17]{lima2023eigenvalue}). For simplicity, denote $a=|\partial \Sigma|_g$ and $A=|\Sigma|_g$. We can find $L^2(\partial \Sigma,g) $-normalized functions $u_0\in V_0(g)$ and $u_1,\ldots,u_m\in V_i(g)$ together with constants $t_j>0$ and $\displaystyle\sum_{j=1}^m t_j=1$  such that
\begin{eqnarray}
\label{eqn:boundary 1}
&\displaystyle \sum_{j=1}^m t_j \left[ \theta_0 \cos^2 r \left( 1- au_0^2 \right)   +\theta_i \sin^2 r \left(1-a u_j^2 \right) \right]=0\ \mathrm{on}\ \partial \Sigma.
\end{eqnarray}
Denote the rescaled functions by
\begin{align}
\label{eqn:eigenfunction rescaled}
 v_0= \sqrt{a}\cos r \cdot u_0,\ v_j= \sqrt{a t_j}\sin r \cdot u_j,\ \mathrm{for}\ 1\leqslant j\leqslant m.
 \end{align} and define $f=\displaystyle\sum_{j=0}^m v_j^2-1$. Notice that \textit{a priori} $f$ is not identically zero on $\partial\Sigma$. However, it is easy to verify that $\displaystyle\int_{\partial\Sigma}f\,dL_g=0$. Now let $\eta$ be the outward-pointing unit normal vector field on $\partial \Sigma$, we have from \eqref{eqn:boundary 1} that 
\begin{align}
\label{eqn:boundary normal derivative}
\dfrac{\partial f}{\partial \eta}=2(\theta_0\cos^2 r+\theta_i \sin^2 r)
\end{align}
is a constant.\\

The interior component of the weak derivative gives
\begin{align}
\label{eqn:interior 1}
g+\dfrac{1}{2}\sum_{j=0}^m \left(-k v_j^2+\vert \nabla^g v_j\vert^2 \right) g-\sum_{j=0}^m dv_j\otimes dv_j=0.
\end{align}
Raising the indices using $g^{-1}$, and taking the trace of above equation yields
\begin{align}
\label{eqn:trace}
k+\left(\dfrac{k}{2}\sum_{j=0}^m\vert \nabla^g v_j\vert^2-\dfrac{k^2}{2}\sum_{j=0}^m v_j^2-\sum_{j=0}^m \vert\nabla^g v_j\vert^2  \right)=0.
\end{align}
Using Equation \eqref{eqn:trace}, we get
\begin{align*}
\Delta_g f=\dfrac{-4k}{k-2}f.
\end{align*}
Note that $\lambda:=\dfrac{-4k}{k-2}<0$ for all $k>2$. Then using the divergence theorem, we get
\begin{align*}
0 \leqslant -\lambda \int_{\Sigma}f^2\,dA_g+ \int_{\Sigma}\vert \nabla^g f\vert_g^2 dA_g&=\int_{\Sigma} (-\Delta_g f\cdot f + \vert \nabla^g f\vert_g^2 )dA_g\\= &\int_{\partial\Sigma} f\cdot  \dfrac{\partial f}{\partial \eta}dL_g= \dfrac{\partial f}{\partial \eta}\int_{\partial\Sigma} f dL_g=0.
\end{align*}
This is only possible, when $f\equiv 0$ on $\Sigma$, i.e., $\displaystyle\sum_{j=0}^m v_j^2=1$ on $\Sigma$. Substituting this back to \eqref{eqn:interior 1} and \eqref{eqn:trace}, we see that one has
$$
\sum_{j=0}^m \vert\nabla^g v_j\vert^2=k \quad \text{and} \quad g=\sum_{j=0}^m dv_j\otimes dv_j.
$$
Therefore, the map
\begin{align*}
\Phi:\Sigma\rightarrow \mathbb{R}^{m+1},\ \Phi(x)=(v_0(x),v_1(x),\ldots,v_m(x))
\end{align*}
is an isometric immersion into $\mathbb{S}^m$.
Since $\dfrac{\partial f}{\partial \eta}=0$ on $\partial\Sigma$, it follows trivially from the steps of the proof of Theorem B in \cite{lima2023eigenvalue}, that $v_0=\cos r$ on $\partial\Sigma$, $\theta_0=-\tan r$ and $\theta_i=\cot r$. This implies $V$ is a FBMI of $\Sigma$ into $\mathbb B^m(r)$.\\

Now we examine the extremal metrics for the functional $\Theta_{r,i}$ the set of conformal admissible metrics for a given $i\geqslant 1$ as before. Let $g\in \mathcal R_{k}(\Sigma)$ be a  Riemannian metric such that $\Theta_{r,i}(g)$ is extremal for the conformal class $[g]\cap\mathcal R_{k}(\Sigma)$. We claim that similarly to the 2-dimensional case in \cite{lima2023eigenvalue,medvedev2025free}, there exists frequency $k$ Steklov eigenfunctions inducing a free boundary harmonic map into $\mathbb B^m(r)$.

Since $g$ is extremal for $\Theta_{r,i}$,  we obtain the following result, analogous to the proof of $2$-dimensional case in \cite{lima2023eigenvalue}. For any positive functions $h\in L^2(\Sigma,g)$ and $h_b\in L^2(\partial\Sigma,g)$ we can find $L^2(\partial\Sigma,g)$-normalized functions $u_0\in V_0(g)$ and $u_i\in V_i(g)$ such that 
\begin{align}
\label{eqn: conformal boundary der}
\int_{\partial\Sigma} h_b \left(\theta_0 \cos^2r \left(1-a u_0^2\right) +\theta_i \sin^2 r \left(1-a u_i^2\right) \right)dL_g=0
\end{align}
and also 
 \begin{align}
\label{eqn: conforaml interior der}
 \resizebox{0.90\textwidth}{!}{$\int_{\Sigma} h \left( a \left( \cos^2 r \left(-\dfrac{k^2}{2}u_0^2+\left(\dfrac{k}{2}-1\right)\vert \nabla^g u_0\vert^2 \right) + \sin^2 r \left( -\dfrac{k^2}{2}u_i^2+\left(\dfrac{k}{2}-1\right)\vert \nabla^g u_i\vert^2 \right) \right) +k \right)dA_g=0.$}
\end{align} 
Therefore, similarly to the proof of FBMI case above, Equation \eqref{eqn: conformal boundary der} implies there exists  $L^2(\partial\Sigma,g)$-normalized eigenfunctions $u_0\in V_0(g)$ and $u_j\in V_i(g)$ for $1\leqslant j\leqslant m $ together with $t_j>0$ with $\sum_j t_j=1$ such that \eqref{eqn:boundary 1} holds.  Rescale these eigenfunctions as in \eqref{eqn:eigenfunction rescaled} to obtain $v_0,v_1,\ldots,v_m$. For $f=\displaystyle\sum_{j=0}^m v_j^2-1$, we also obtain  \eqref{eqn:boundary normal derivative} as in the FBMI case.
Moreover, it follows from equation \eqref{eqn: conforaml interior der} that \eqref{eqn:trace} also holds. Therefore, we can also deduce that $\displaystyle\sum_{j=0}^m v_j^2=1$. We see that map $\Phi: \Sigma\rightarrow \mathbb{R}^{m+1}$ defined by $x\mapsto (v_0(x),v_1(x),\ldots,v_m(x))$ has image inside $\mathbb{S}^m$. Finally, combining equations \eqref{eqn:boundary 1} and $\dfrac{\partial f}{\partial \eta}=0$ on $\partial\Sigma$ together, we have $v_0=\cos r$ on $\partial\Sigma$, see also proof of Theorem B in \cite{lima2023eigenvalue}.
\end{proof}

We now prove Theorem \ref{2.3} that serves as a converse statement to Theorem \ref{2.2}, under an additional assumption. Different choices of normalized $u_j\in V_i(g)$ in \eqref{sphere functioanl dev} give rise to (one-sided) derivatives $\theta_i(g_t)$ at $t=0$. Therefore, to obtain information on the one-sided derivatives of $\theta_i$ at $t=0$ that is possibly non-differentiable, we make the following assumption: \textit{For some given fixed $i\geqslant 1$ and Riemannian metric $g$, we either have $\theta_i<\theta_{i+1}$ or $\theta_i>\theta_{i-1}$}.

\begin{proof}[Proof of Theorem \ref{2.3}]
We prove the first statement following the steps in \cite[prop. 3.3]{medvedev2025free} as the second statement is totally analogous. Applying the normal derivative $\partial\eta$ to the equation $$\displaystyle\sum_{j=0}^m v_j^2=1,$$ we obtain
\begin{align*}
\theta_0 v_0^2+\theta_i(1-v_0^2)=0\ \mathrm{on}\ \partial\Sigma.
\end{align*}
Since $\theta_i-\theta_0>0$, we see that $v_0^2$ is a constant on $\partial\Sigma$. If we have $v_0^2=0$ on $\partial\Sigma$, the equation above implies $\theta_i=0$. In addition, the fact that $g\in\mathcal{R}_{k}(\Sigma)$ gives $v_0\equiv 0$ on $\Sigma$ and $\theta_0=0$. This leads to a contradiction. Similarly, we obtain that $\displaystyle\sum_{j=1}^m v_j^2$ is a non-zero constant on $\partial\Sigma$. Therefore, there is some $0<r<\dfrac{\pi}{2}$ such that 
\begin{align*} 
v_0^2=\cos^2 r,\ \sum_{j=1}^m v_j^2=\sin^2 r\ \mathrm{on}\ \partial\Sigma.
\end{align*}
Now rescale these functions, we obtain $L^2(\partial\Sigma,g^b)$-normalized eigenfunctions $u_0,u_1,\ldots,u_m$, and set the constants 
\begin{align*}
t_j=\dfrac{1}{a\sin^2 r}\Vert v_j\Vert^2_{L^2(\partial\Sigma,g^b)},\ \forall 1\leqslant j\leqslant m.
\end{align*}
We see that $t_j>0$ and $\displaystyle\sum_{j=1}^m t_j=1$.
Analogously to \cite[prop. 3.3]{medvedev2025free}, we have
\begin{align}
\sum_{j=1}^m t_j Q_r^s(u_0,u_j)=0.
\end{align}
Thus, in summation above, there are non-negative and non-positive terms of the form $Q_r^s(u_0,u_j)$. Note that $\theta_0(g_t)$ is always simple. First, suppose $\theta_i (g_t) $ is the smallest eigenvalue for all $\theta_l(g_t)$ with $\theta_l(g_0)=\theta_i(g_0)$. Since $\sin^2 r>0$, there exists $Q_r^s(u_0,u_j^+)\geqslant 0$ implies 
\begin{align*}
\lim_{t\rightarrow 0^-} \dfrac{\Theta_{r,i}(g_t)-\Theta_{r,i}(g_0)}{t}\geqslant Q_r^s(u_0,u_j^+)\geqslant 0.
\end{align*}
Similarly, for $Q_r^s(u_0,u_j^-)\leqslant 0$, we obtain
\begin{align*}
\lim_{t\rightarrow 0^+} \dfrac{\Theta_{r,i}(g_t)-\Theta_{r,i}(g_0)}{t}\leqslant Q_r^s(u_0,u_j^+)\leqslant 0
\end{align*}
Thus we see $g$ is an extremal metric for the functional $\Theta_{r,i}$. Similarly, if $\theta_i(g_t)$ is the largest eigenvalue for all $\theta_j(g_t)$ with $\theta_j(g)=\theta_i(g)$, we can show that $g$ is also extremal for $\Theta_{r,i}$. This completes the proof.
\end{proof}
\subsection{Coefficients of the sphere FBMI functional}\label{subsec:coeff}
We provide a justification why the functional defined in \eqref{eqn:sphere functional} is of the same form as the functional in \cite{lima2023eigenvalue} for the 2-dimensional case. One natural question is whether there are functionals in a similar form that give rise to critical metrics, induced by FBMI in a geodesic ball in the standard sphere? We consider the functionals in the following form:
\begin{align}
\label{eqn: sphere gen functional guess}
\Theta_r(\Sigma,g)=(\cos^2 r\theta_0+\sin^2 r\theta_1)\alpha_1 a^{\beta_1}+\alpha_2 A^{\beta_2}
\end{align}
Here $a=\vert \partial\Sigma\vert_g$ and $A=\vert\Sigma\vert_g$ for simplicity, and $\alpha_1,\alpha_2,\beta_1,\beta_2$ are given constants. Note that if we set $i=1$, then \eqref{eqn:sphere functional} belongs to the class of functionals given by \eqref{eqn: sphere gen functional guess}. We would like to investigate, whether there are other choices of the coefficients $\alpha_l,\beta_l$ such that the functional in the form of \eqref{eqn: sphere gen functional guess} would have critical metrics, induced by FBMI in a geodesic ball in the standard sphere.\\

We want the critical points of \eqref{eqn: sphere gen functional guess} to include all the metrics whose eigenfunctions corresponding to the Steklov eigenvalues $\theta_0,\theta_1$ (of frequency $k$) induce FBMI embeddings into unit spherical caps. Let $\mathbb B^k(r)$ be the spherical cap in $\mathbb{S}^k$ at $(1,0,\ldots,0)$. The coordinate functions are spherical harmonics of degree $1$. Therefore, $\Delta_g x_i-k x_i=0$, for all $i=0,\ldots,k$. Regarding the Steklov eigenvalues and eigenfunctions, theorem 3 of \cite{lima2023eigenvalue} and its proof for $2$-dimensional spherical caps can be generalized to higher dimensions, and the detailed steps are given by the second author in the proof of proposition 5.6 of \cite{medvedev2025free}. We collect the results and obtain the following.
\begin{proposition}
\label{claim:sphere eigen}
 For the spherical cap $\mathbb B^k(r)\subset\mathbb{S}^k$ centered at  $(1,0,\ldots,0)$ with $r<\dfrac{\pi}{2}$, the eigenvalue $\theta_0=-\tan r$ has eigenfunction $x_0$, and the eigenvalue $\theta_1=\cot r$ has eigenspace generated by $x_1,\ldots,x_k$. For a geodesic ball $\mathbb{B}^k(r)\subset \mathbb{H}^k$ centered at $(1,0,\ldots,0)$, the eigenvalue $\omega_0=\tanh (r)$ has eigenspace generated by $x_0$, and $\omega_1=\coth (r)$ has eigenspace generated by $x_1,\ldots,x_k$.
\end{proposition}

The set of extremal metric of \eqref{eqn: sphere gen functional guess} should include the canonical metric on $\mathbb B^k(r)$. The weak derivative of 
\begin{align*}
\Theta_r(\Sigma,g)=(\cos^2 r\theta_0+\sin^2 r\theta_1)\alpha_1 a^{\beta_1}+\alpha_2 A^{\beta_2}
\end{align*}
is given by the following quadratic form
\begin{align*}
Q(u_0,u_1)=\alpha_1 a^{\beta_1-1}\left[ \cos^2 r(\dot{a}\beta_1\theta_0+a\dot{\theta_0})+\sin^2 r (\dot{a}\beta_1\theta_1+a\dot{\theta_1})  \right]+\alpha_2\beta_2 A^{\beta_2-1},
\end{align*}
where $u_0$ and $u_1$ are $L^2(\partial\Sigma,g^b)$-normalized eigenfunctions of $\theta_0$ and $\theta_1$, respectively. They are encoded in the formulas of $\dot{\theta_0}$ and $\dot{\theta_1}$.
Thus, the boundary component of $Q(u_0,u_1)$ is given by
\begin{align*}
Q_b=\alpha_1 a^{\beta_1-1}\cdot \dfrac{1}{2}\int_{\partial\Sigma}\left[ \theta_0\cos^2 r (\beta_1-a u_0^2)+\theta_1\sin^2 r(\beta_1-au_1^2) \right] \langle g^b,\dot{g^b}\rangle dL
\end{align*}
From Proposition~\ref{claim:sphere eigen}, we see that $\theta_0=-\tan r$ and $\theta_1=\cot r$. In addition, $u_0$ is a multiple of $x_0$, so it is a constant on $\partial\Sigma$. Suppose there are  $L^2(\partial\Sigma,g^b)$-normalized eigenfunctions $u_0$ of $\theta_0$, $u_1,\ldots,u_l$ of $\theta_1$ and positive constants $\lbrace t_j\rbrace_{j=1}^l$ with $\displaystyle\sum_{j=1}^l t_j=1$ such that 
\begin{align*}
\sum_{j=1}^l Q(u_0,u_j)=0.
\end{align*}
This gives 
\begin{align*}
\sum_{j=1}^l t_j u_j^2=u_0^2=\dfrac{1}{a} \equiv const,\ \partial\Sigma.
\end{align*}
When restricting to the $(x_1,\ldots,x_k)$ coordinate, we see $\partial\Sigma$ is the level set of 
the equation $x_1^2+\ldots+x_k^2=\sin^2 r$. We know that each $u_j$ is a linear combination of $x_1,x_2,\ldots,x_k$.
By assumption all $t_j>0$. Thus, the only possibility is $l=k$ and the change of coordinate from $(\sqrt{t_1} u_1,\ldots,\sqrt{t_k} u_k)$ to $(x_1,\ldots,x_k)$ is a linear change of basis of the from $cO$, where $c$ is constant scaling and $O\in O_k(\mathbb{R})$. Thus, all functions $\sqrt{t_j}u_j$ have the same $L^2(\partial\Sigma,g^b)$-norm. Now use the fact that each $u_j$ is $L^2(\partial\Sigma,g^b)$-normalized, the assumption $\displaystyle\sum_{j=1}^k t_j=1$ yields
$$t_j=\dfrac{1}{k},\ \forall 1\leqslant j\leqslant k.$$ Using this, the substitution into the interior integral component yields

\begin{align}
\alpha_1 a^{\beta_1-1}=\dfrac{1}{2}\alpha_2 \beta_2 A^{\beta_2-1}.
\end{align}
The constants $\alpha_l,\beta_l$ have to depend only on $k$, not on a specific manifold. Therefore, $\dfrac{A^{\beta_2-1}}{a^{\beta_1-1}}$ has to be a fixed constant for all $k$-dimensional extremal metrics. This is only possible, when $\beta_2=\beta_1=1$. Without loss of generality, if we take $\alpha_1=0$, we get $\alpha_2=2$.
Thus, for any $k$-dimensional case, the only possible functional of the form \eqref{eqn: sphere gen functional guess} can only be
\begin{align}
\Theta_r(\Sigma,g)=(\cos^2 r \theta_0+\sin^2 r \theta_1)a+2A,
\end{align}
which is exactly the same as the $2$-dimensional case.

For the hyperbolic case, an analogous statement can also be proved. 

Following the steps in the above spherical case, we examine the functionals in the following form
\begin{align*}
\Omega_r(\Sigma,g):=(-\omega_0 \cosh^2 r+\omega_1 \sinh^2 r)\vert\partial\Sigma\vert_g^{\beta_1} +\alpha_2 \vert \Sigma\vert_g^{\beta_2}.
\end{align*}
If we want its critical metrics to include all those inducing FBMI via corresponding eigenfunctions, then one also need to have
$\beta_1=\beta_2=1$ and $\alpha_2=2$.

\section{FBMI in hyperbolic spaces}\label{sec:hyper}
For the FBMI on hyperbolic spaces, we can develop results similar to the spherical case following the steps in \cite{medvedev2025free}. Before introducing the theorem in the hyperbolic case, let us recall the following facts and notations.

 Let $\Sigma^k$ a compact $k$-dimensional manifold, and $\mathcal{R}(\Sigma)$ be the set of all Riemannian metrics on $\Sigma^k$. Unlike the spectral problem~\eqref{sys:dsphere} in the previous section, by the divergence theorem, for any metric $g\in\mathcal{R}(\Sigma)$, the following PDE system has only trivial solutions
\begin{equation}
\left\{ \begin{aligned} 
  ku+ \Delta_g u =0 \ &\mathrm{in}\ \Sigma,\\
   u =0\ &\mathrm{on}\ \partial\Sigma.
\end{aligned} \right.
\end{equation}
For metrics $g\in\mathcal{R}(\Sigma)$, the Steklov eigenvalues of frequency $-k$ consist of the increasing sequence of real numbers $\{\omega_i\}$ such that the following system has non-trivial solutions.
\begin{equation}
\left\{ \begin{aligned} 
  ku+ \Delta_g u =0 \ &\mathrm{in}\ \Sigma,\\
   \dfrac{\partial u}{\partial \eta} =\omega_i u\ &\mathrm{on}\ \partial\Sigma.
\end{aligned} \right.
\end{equation}

 For any $r>0$ and $i\geqslant 1$, the hyperbolic FBMI functional $\Omega_{r,i}$ defined on $\mathcal{R}(\Sigma)$ \textit{in the same exact form} introduced in \cite{medvedev2025free}, is given by the following formula, regardless the dimension $k$ of $\Sigma$:
\begin{align}
\Omega_{r,i}(\Sigma,g):=(-\omega_0 \cosh^2 r+\omega_i \sinh^2 r)\vert\partial \Sigma\vert_g + 2\vert \Sigma\vert_g.
\end{align}

Now let $\tilde{\mathcal{R}}(\Sigma)$ be the subset of $\mathcal{R}(\Sigma)$ such that the PDE system of the following type only has zero solution.
\begin{equation}
\left\{ \begin{aligned} 
  \Delta_g f-\left( \dfrac{4k}{k-2}\right)f=0  \ &\mathrm{in}\ \Sigma,\\
   \dfrac{\partial f}{\partial \eta}=c \ &\mathrm{on}\ \partial\Sigma.
\end{aligned} \right.
\end{equation}
where $c\in \mathbb{R}$ is any real constant. We prove the following theorem for the hyperbolic FBMI functionals.

\begin{theorem}
Let $\Sigma^k$ with $k\geqslant 2$ be a $k$-dimensional compact smooth manifold with boundary. For $r>0$, let $\mathbb B^m(r)$ be the ball of radius $r$ in $\mathbb{H}^m$ centered at the point $(1,0,\ldots,0)$ in the Minkowski space $\mathbb{M}^{1,m}$, and $V_i(g)$ be the eigenspace of the Steklov eigenvalue $\omega_i$ of frequency $-k$. Then the following is true:
\begin{itemize}
\item If $g\in \tilde{\mathcal{R}}(\Sigma)$ is an extremal metric for $\Omega_{r,i}$ defined on $\mathcal{R}(\Sigma)$, then there exist Steklov eigenfunctions with frequency $-k$ $v_0\in V_0(g)$ and $v_1,\ldots,v_m\in V_i(g)$ such that the following map
$$\Phi:\Sigma\rightarrow \mathbb{M}^{1,m},\ x\mapsto (v_0(x),v_1(x),\ldots,v_m(x))$$
is in fact an FBMI into $\mathbb B^m(r)$.
\item If $g\in \tilde{\mathcal{R}}(\Sigma)$ is an extremal metric for $\Theta_{r,i}$ on a Riemannian conformal class $[g]$, then there exist Steklov eigenfunctions with frequency $-k$  $v_0\in V_0(g)$ and $v_1,\ldots,v_n\in V_i(g)$ such that the map $\Phi$ is an FBHM into $\mathbb B^m(r)$.
\end{itemize}
\end{theorem}
\begin{proof}
The case of $k=2$ has already been proved in \cite[prop. 3.2]{medvedev2025free}. In this proof, we always assume $k>2$.

We show the first statement of theorem, following the steps in Section \ref{sec:2.1}, and the case that restricts to the conformal class is completely analogous.
Let $g_t\in \mathcal{R}(\Sigma)$ with $g_0=g$ be a smooth family of metrics. The weak derivative $\dfrac{d}{dt}\omega_i$ is well-defined a.e. and given by
\begin{align*}
\dfrac{d}{dt}\omega_i=  & \int_{\Sigma} \left\langle \dfrac{1}{2}\left(k u_i^2(t)+ \vert \nabla^t u_i(t) \vert_{g(t)}^2 \right)g(t)- du_i(t)\otimes du_i(t), \dot{g}(t) \right\rangle_{g(t)} dA_t\\
& -\omega_i(g(t))\int_{\partial \Sigma} \dfrac{1}{2} u_i^2(t) \left\langle g^b(t),\dot{g}^b(t) \right\rangle_t dL_t,
\end{align*}
where $u_i(t)$ is an eigenfunction of $\omega_i(t)$, normalized with respect to the space $L^2(\partial\Sigma, g^b)$ as before. Similar to $\Theta_{r,i}$, the weak derivative of $\Omega_{r,i}$ is also given by a sum of two integrals, and at $t=0$ the integrand of the boundary component is 
\begin{align}
\dfrac{1}{2}\langle \dot{g^b},g^b\rangle \left( \omega_0 \cosh^2 r(a u_0^2-1)+ \omega_i \sinh^2 r(au_i^2-1 )    \right).
\end{align}
where $a=\vert\partial\Sigma\vert_g$ as before.
The integrand of the interior component is
\begin{align}
\label{eqn:interior hyperbolic integrand}
\langle \dot{g}, g-\cosh^2 r \tau_H(u_0)+\sinh^2 r\tau_H(u_i)\rangle,
\end{align}
where the expression $\tau_H$ is given by
\begin{align*}
\tau_H(u)=a\left(\dfrac{1}{2}(k u^2+\vert \nabla^g u\vert_g)-du\otimes du\right).
\end{align*}

Lemma \ref{lemma:sphere vanish derivative} and the argument involving Hahn-Banach theorem also work for the functional $\Omega_{r,i}$. Similar to the spherical case, we can find $L^2(\partial \Sigma,g) $-normalized functions $u_0\in V_0(g)$ and $u_1,\ldots,u_m\in V_i(g)$ together with constants $t_j>0$ and $\sum_{j=1}^m t_j=1$  such that
\begin{eqnarray}
\label{eqn:hboundary 1}
& \sum_{j=1}^m t_j \left[- \omega_0 \cosh^2 r \left( 1- au_0^2 \right)   +\omega_i \sinh^2 r \left(1-a u_j^2 \right) \right]=0\ \mathrm{on}\ \partial \Sigma.
\end{eqnarray}
Denote the rescaled functions by
\begin{align}
\label{eqn:heigenfunction rescaled}
 v_0= \sqrt{a}\cosh r \cdot u_0,\ v_j= \sqrt{a t_j}\sinh r \cdot u_j\ \mathrm{for}\ 1\leqslant j\leqslant m.
 \end{align} 
 and define $f=-v_0^2+\sum_{j=1}^m v_j^2+1$. It is not hard to see that $\int_{\partial\Sigma}f\,dL_g=0$. Also, \eqref{eqn:hboundary 1} implies that 
\begin{align}
\label{eqn:hboundary normal derivative}
\dfrac{\partial f}{\partial \eta}=2(-\omega_0\cosh^2 r+\omega_i \sinh^2 r)
\end{align}
is a constant. 
Using \eqref{eqn:interior hyperbolic integrand}, the interior component of the weak derivative of $\Omega_{r,i}$ gives
\begin{align}
\label{eqn:hinterior 1}
g+\dfrac{1}{2}\sum_{j=1}^m \left(k v_j^2+\vert \nabla^g v_j\vert^2 \right) g-\sum_{j=1}^m dv_j\otimes dv_j-\dfrac12\left(k v_0^2+\vert \nabla^g v_0\vert^2 \right) g+ dv_0\otimes dv_0=0.
\end{align}
Raise the indices using $g^{-1}$ and take the trace of above equation yields
\begin{align}
\label{eqn:htrace}
k+\dfrac{k-2}{2}\left(-\vert \nabla^g v_0\vert^2+\sum_{j=1}^m\vert \nabla^g v_j\vert^2\right)+\dfrac{k^2}{2}\left(-v_0^2+\sum_{j=1}^m v_j^2\right)=0.
\end{align}
Using equation \eqref{eqn:htrace}, we get
\begin{align}
\Delta_g f=\dfrac{4k}{k-2}f.
\end{align}
Since $g\in \tilde{\mathcal{R}}(\Sigma)$ by assumption, we must have
$f=0$ on $\Sigma^k$. It follows that the map $\Phi$ has image in $\mathbb{H}^m$. Substitute $f=0$ back to \eqref{eqn:hinterior 1} and \eqref{eqn:htrace}, we also find the map
$\Phi$ is an isometric immersion into $\mathbb{H}^m$. The remaining steps showing $\Phi$ has image and is a FBMI in fact into $B^m(r)$ are completely the same as the spherical case. This completes the proof.
\end{proof}
In addition to the theorem above, we see that the proof of Theorem \ref{2.3} can be applied to functional $\Omega_{r,i}$ at little cost. This yields the following converse statement.
\begin{theorem}
Given $i\geqslant 1$, suppose for some $g\in \mathcal{R}(\Sigma^k)$, we have either $\theta_i(g)<\theta_{i+1}(g)$ or $\theta_i(g)>\theta_{i-1}(g)$. Let $V_i(g)$ be the eigenspace of the Steklov eigenvalue $\omega_i$ of frequency $-k$.
\begin{itemize}
\item Suppose there exists $v_0\in V_0(g)$ and independent $v_1,\ldots,v_m\in V_i(g)$ such that
\begin{enumerate}
\item $-dv_0\otimes dv_0+\displaystyle\sum_{j=1}^m dv_j\otimes dv_j=g$
\item $-v_0^2+\displaystyle\sum_{j=1}^m v_j^2=-1$
\end{enumerate}
Then there is some $r>0$ such that $g$ is an extremal metric for $\Omega_{r,i}$ in the set $\mathcal{R}(\Sigma)$.
\item If there exist $v_0\in V_0(g)$ and independent $v_1,\ldots,v_m\in V_i(g)$ such that $-v_0^2+\displaystyle\sum_{j=1}^m v_j^2=-1$, the metric $g$ is extremal for $\Omega_{r,i}$ in the set $[g]$ for some $r>0$.
\end{itemize}
\end{theorem}

\section{Alternative forms of FBMI characterizing functionals}
\label{sec:alt functional}
Here we give the functionals that (potentially) characterize the FBMI of compact manifolds into balls in $\mathbb{S}^m_+$ or $\mathbb{H}^m$. 

Let $\Sigma^k$ be a $k$-dimensional compact manifold with piecewise smooth boundary $\partial\Sigma$ as before.
We introduce the following two functionals on the set $\mathcal R(\Sigma)$ of Riemannian metrics on $\Sigma$:
 \begin{align}
 \label{eqn:alt sphere functional}
 \Xi_{r,i}^+(\Sigma,g):=\min\left\lbrace\lambda_0(g,-\tan(r)),\lambda_i(g,  \cot(r))\right\rbrace|\Sigma|_g^{2/k},\ i\geqslant 1, \, 0<r<\frac{\pi}{2}
\end{align}
 and 
  \begin{align}
  \label{eqn:alt hyperbolic functional}
 \Xi_{r,i}^-(\Sigma,g):=\min\{\lambda_0(g, \tanh(r)),\lambda_i(g, \coth(r))\}|\Sigma|_g^{2/k},\, i\geqslant 1, \, r>0.
 \end{align}
 Here $\lambda_i(g,\sigma)$ is the $i$-th eigenvalue, arranged in increasing order, of the Robin problem with $\sigma\in\mathbb R$ fixed
 \begin{align}\label{sys:robin}
 \begin{cases}
 \Delta_gu=\lambda_i(g,\sigma)u\,&\text{ in}\, \Sigma,\\
 \dfrac{\partial u}{\partial \eta}=\sigma u\, &\text{ on} \,\partial\Sigma.
 \end{cases}
 \end{align}
These two functionals are naturally associated to FBMI in the following sense. Let $\Phi\colon \Sigma \to \mathbb B^m(r)\subset \mathbb S^m_+$ be a free boundary immersion, given by the $\lambda_0(g,-\tan(r))$ and $\lambda_i(g,\cot(r))$-eigenfunctions, then the energy $E[\Phi]$ satisfies $2E[\Phi]= \Xi_{r,i}^+(\Sigma,g)$. Similarly, let $\Phi\colon \Sigma \to \mathbb B^m(r)\subset \mathbb H^m$ be a free boundary immersion, given by the $\lambda_0(g,\tanh(r))$ and $\lambda_i(g,\coth(r))$-eigenfunctions, then the energy $E[\Phi]$ satisfies $2E[\Phi]= \Xi_{r,i}^-(\Sigma,g)$. \\

 It is not hard to see that both functionals $\Xi_{r,i}^+(\Sigma,g)$ and $\Xi_{r,i}^-(\Sigma,g)$ are bounded from above on $\mathcal R(\Sigma)$. Indeed, it follows from their definitions that
 $$
 \Xi_{r,i}^+(\Sigma,g) \leqslant \lambda_i(g,  \cot(r))|\Sigma|_g^{2/k} \quad \text{and} \quad  \Xi_{r,k}^-(\Sigma,g) \leqslant \lambda_i(g,  \coth(r))|\Sigma|_g^{2/k}.
 $$
 Further, we apply~\cite[Proposition 2.6]{hassannezhad2021nodal}, which implies that $\lambda_i(g, \cot(r))<\lambda_i(g,0)=\lambda^N_i(g)$, where $\lambda^N_i(g)$ is the $i$-th Neumann eigenvalue. Similarly, $\lambda_i(g, \coth(r))<\lambda^N_i(g)$. Thus,
  $$
 \Xi_{r,i}^+(\Sigma,g) \leqslant \Lambda^N_i(\Sigma,g) \quad \text{and} \quad  \Xi_{r,i}^-(\Sigma,g) \leqslant \Lambda^N_i(\Sigma,g),
 $$
 where $\Lambda^N_i(\Sigma,g)$ denotes the $i$-th normalized Neumann eigenvalue. It has been proved in~\cite{li1982new,korevaar1993upper} that $\Lambda^N_i(\Sigma,g)$ is bounded from above in the conformal class $[g]$ for any $g\in \mathcal R(\Sigma)$ if $k>2$, and bounded from above on $\mathcal R(\Sigma)$ if $k=2$. Moreover, in~\cite{li1980estimates}, it has been shown that $\Lambda^N_i(\Sigma,g)$ is bounded from above on the set of metrics of nonnegative Ricci curvature (or more generally Ricci curvature bounded below and diameter bounded above). Summing up, we obtain the following. 
 
 \begin{proposition}\label{prop:bound}
 The functionals $\Xi_{r,i}^+(\Sigma^k,g)$ and $\Xi_{r,i}^-(\Sigma^k,g)$ are bounded from above on $\mathcal F\subset \mathcal R(\Sigma^k)$ in the following cases:
 \begin{itemize}
 \item $k=2$ and $\mathcal F=\mathcal R(\Sigma)$;
 \item $k>2$ and $\mathcal F=[g]$ for any $g\in \mathcal R(\Sigma)$;
 \item $k>2$ and $\mathcal F$ is the set of metrics with Ricci curvature bounded below and diameter bounded above.
 \end{itemize}
 \end{proposition}
 
 Now we focus on the case $k=2$ and obtain explicit upper bounds of these functionals as in Theorem~\ref{thm:bound}, whose proof is given as follows.

\begin{proof}[Proof of Theorem~\ref{thm:bound}]
 Take a conformal map $u\colon (\Sigma,g) \to \mathbb B^2(r)$, where $\mathbb B^2(r)$ denotes the geodesic disk in $\mathbb S^2_+$ of radius $r$ centered at $(1,0,0)$. There always exists a map $u$ with degree $\deg(u)\leqslant 2(\gamma+l)$ (see~\cite{gabard2006representation}). Let $x_i,\,i=0,1,2$ be the coordinate functions of $\mathbb R^3$. Consider $u_i=u\circ x_i,\,i=0,1,2$. One has
 \begin{align}\label{eq:boundary}
 u_0=\cos r,\quad u_1^2+u_2^2=\sin^2r\, \text{ along}\, \partial \Sigma.
 \end{align}
Using the Hersch trick, one can always assume that
$$
\int_\Sigma u_i\phi dA_g=0, \, i=1,2,
$$
where $\phi$ is an eigenfunction of the eigenvalue $\lambda_1(g,\sigma)$ for ~\eqref{sys:robin} with $\sigma=\cot(r)$. Since the map $u$ is conformal, then $u^*\delta=f^2 g$, where $f$ is a positive function on $\Sigma$ and $\delta$ is the canonical metric on $\mathbb B^2(r)$. It is not hard to verify that
 $$
 \sum_{i=0}^2|\nabla^gu_i|^2_g=2f^2.
 $$
 It follows from the variational characterization that
 $$
\lambda_0(g,-\tan(r)) \int_\Sigma u_0^2\,dA_g \leqslant \int_\Sigma|\nabla^g u_0|^2_g\,dA_g+\tan(r)\int_{\partial \Sigma}u^2_0\,dL_g,
 $$
 and
  $$
 \lambda_1(g,\cot(r)) \int_\Sigma u_i^2\,dA_g \leqslant \int_\Sigma|\nabla^g u_0|^2_g\,dA_g-\cot(r)\int_{\partial \Sigma}u^2_i\,dL_g,\, i=1,2.
 $$
 Summing these inequalities, we get
 \begin{align}\label{eq:main}
  \Xi_{r,1}^+(\Sigma,g)&=\min\left\lbrace\lambda_0(g,-\tan(r)),\lambda_1(g,  \cot(r))\right\rbrace|\Sigma|_g \\ \nonumber &=\min\left\lbrace\lambda_0(g,-\tan(r)),\lambda_1(g,  \cot(r))\right\rbrace \int_\Sigma \sum_{i=0}^2u_i^2\,dA_g  \\ \nonumber &\leqslant  \lambda_0(g,-\tan(r)) \int_\Sigma u_0^2\,dA_g+\lambda_1(g,\cot(r)) \int_\Sigma \sum_{i=1}^2u_i^2\,dA_g \\ \nonumber &\leqslant \sum_{i=0}^2\int_\Sigma|\nabla^g u_i|^2_g\,dA_g=2\int_\Sigma f^2\,dA_g\leqslant 4\pi(1-\cos r)(\gamma+l),
 \end{align}
 where we used $\sum_{i=0}^2 u_i^2=1$ in the first inequality and also ~\eqref{eq:boundary} in the second inequality, which yields that $$\tan(r)\int_{\partial \Sigma}u^2_0\,dL_g=\cot(r)\int_{\partial \Sigma}\sum_{i=1}^2u^2_i\,dL_g .$$ The inequality is proved.\\
 
 Let $\Sigma$ be a topological disk. Then $\gamma=0, \, l=1$ and one gets
 $$
  \Xi_{r,1}^+(\Sigma,g)\leqslant 4\pi(1-\cos r).
 $$
 Let equality hold. Then we also have equalities at each step in~\eqref{eq:main}. This immediately implies that $\lambda_0(g,-\tan(r))=\lambda_1(g,\cot(r))$. Moreover,
 \begin{align*}
 0=\Delta_g\sum_{i=0}^2u_i^2=2&\sum_{i=0}^2u_i\Delta_gu_i-2\sum_{i=0}^2|\nabla^gu_i|^2_g\\=&2\lambda_0(g,-\tan(r))u_0^2+2\lambda_1(g,\cot(r))\sum_{i=1}^2u_i^2-2\sum_{i=0}^2|\nabla^gu_i|^2_g,
 \end{align*}
 which yields
 \begin{align*}
2f^2=\sum_{i=0}^2|\nabla^gu_i|^2_g&=\lambda_0(g,-\tan(r))u_0^2+\lambda_1(g,\cot(r))\sum_{i=1}^2u_i^2\\&=\lambda_0(g,-\tan(r))=\lambda_1(g,\cot(r)),
 \end{align*}
 hence $f\equiv const$. Then the metric $g$ is homothetic to $\delta$. In this case, we have $|\Sigma|_g=2\pi f^2(1-\cos r)$. Since $ \Xi_{r,1}^+(\Sigma,g)=4\pi(1-\cos r)$,
 $$
 \lambda_0(g,-\tan(r))=\lambda_1(g,\cot(r))=\dfrac{2}{f^2}.
 $$
 But then $f\equiv 1$, i.e., $\Sigma=\mathbb B^2(r)$.
  \end{proof}
  
Using the approach developed in the paper~\cite{medvedev2025free}, when studying the behaviour of functionals $\Theta_{r,i}(\Sigma^2,g)$ and $\Omega_{r,i}(\Sigma^2,g)$, we prove the following.

\begin{proposition}\label{prop:unbound}
The functionals $\Xi_{r,i}^+(\Sigma^2,g)$ and $\Xi_{r,i}^-(\Sigma^2,g)$ are not bounded from below in $[g]$ for any $g\in\mathcal R(\Sigma)$.
\end{proposition}

\begin{proof}
Let $(\partial\Sigma)_{\varepsilon^2}$ be the $\varepsilon^2$-neighbourhood of $\partial\Sigma$ in $(\Sigma,g)$. Consider the following family $(\varphi_{\varepsilon})_{\varepsilon}$ of non-negative functions on $\Sigma$ such that
\begin{gather*}
 supp(\varphi_{\varepsilon}) \subset (\partial\Sigma)_{\varepsilon^2},\quad \varphi_{\varepsilon} < -\log \varepsilon ~\text{on}~\Sigma, \quad \text{and}~ \varphi_{\varepsilon}=-\log{\varepsilon}~\text{on}~\partial\Sigma.
\end{gather*}
Consider the following family of conformal metrics $(g_{\varepsilon})_{\varepsilon}\in [g]$, where $g_{\varepsilon}:=e^{2\varphi_{\varepsilon}}g$. Then for the Robin problem
$$
 \begin{cases}
 \Delta_{g_{\varepsilon}}u=\lambda_i(g_\varepsilon,\sigma) u\,&\text{ in}\, \Sigma,\\
 \dfrac{\partial u}{\partial \eta_{\varepsilon,\alpha}}=\sigma u\, &\text{ on} \,\partial\Sigma
 \end{cases}
$$  
with $\sigma>0$, the variational characterization yields (see e.g.~\cite[Formula (3.1.17)]{levitin2023topics})
$$
\lambda_i(g_\varepsilon,\sigma) \leqslant \max_{0\neq u\in W}\frac{\displaystyle\int_\Sigma|\nabla^{g_\varepsilon}u|^2_{g_{\varepsilon}}\,dA_{g_{\varepsilon}}-\sigma\int_{\partial \Sigma}u^2\,dL_{g_{\varepsilon}}}{\displaystyle\int_\Sigma u^2\,dA_{g_{\varepsilon}}},
$$
where $W$ is an $(i+1)$-dimensional space in $H^1(\Sigma,dv_{g_{\varepsilon}})$, which is equivalent to $H^1(\Sigma,dv_g)$. Using the conformal invariance of the Dirichlet energy for surfaces and the observation that
$$
\int_\Sigma u^2\,dA_{g_{\varepsilon}}=\int_{\Sigma\setminus (\partial\Sigma)_{\varepsilon^2}} u^2\,dA_{g}+\int_{(\partial\Sigma)_{\varepsilon^2}} u^2\,e^{2\varphi_{\varepsilon}}dA_{g}>\int_\Sigma u^2\,dA_{g},
$$
since $\varphi_{\varepsilon}\geqslant 0$, we get
$$
\lambda_i(g_\varepsilon,\sigma) < \max_{0\neq u\in W}\frac{\displaystyle\int_\Sigma|\nabla^{g}u|^2_{g}\,dA_{g}-\frac{1}{\varepsilon}\sigma\int_{\partial \Sigma}u^2\,dL_{g}}{\displaystyle\int_\Sigma u^2\,dA_{g}}.
$$
Taking the minimum over all $W\subset H^1(\Sigma,dv_g)$ of dimension $i+1$ from both parts of this inequality and using the variational characterization~\cite[Formula (3.1.17)]{levitin2023topics} once again, we get
\begin{align}\label{ineq:sigma}
\lambda_i(g_\varepsilon,\sigma) \leqslant \lambda_i(g,\dfrac{1}{\varepsilon}\sigma).
\end{align}
Consider $\Xi_{r,i}^+(\Sigma,g)$. The case of $\Xi_{r,i}^-(\Sigma,g)$ is absolutely similar. From its definition, we have
$$
\Xi_{r,i}^+(\Sigma,g_{\varepsilon}) \leqslant \lambda_i(g_{\varepsilon}, \cot(r))|\Sigma|_{g_{\varepsilon}}.
$$
Take $\sigma=\cot(r)$ in~\eqref{ineq:sigma}, then the previous inequality becomes
\begin{align}\label{ineq:Xi+}
\Xi_{r,i}^+(\Sigma,g_{\varepsilon}) \leqslant \lambda_i(g, \dfrac{1}{\varepsilon}\cot(r))|\Sigma|_{g_{\varepsilon}}
\end{align}
It is not hard to see that $|\Sigma|_{g_{\varepsilon}} \to |\Sigma|_{g}$, as $\varepsilon\to 0$. Indeed,
\begin{align}\label{eq:area}
&|\Sigma|_{g}<|\Sigma|_{g_{\varepsilon}} =|\Sigma\setminus (\partial\Sigma)_{\varepsilon^2}|_{g}+|(\partial\Sigma)_{\varepsilon^2}|_{g_{\varepsilon}}=|\Sigma\setminus (\partial\Sigma)_{\varepsilon^2}|_{g}+\int_{(\partial\Sigma)_{\varepsilon^2}}e^{2\varphi_{\varepsilon}}dA_{g}\\\nonumber &\leqslant |\Sigma\setminus (\partial\Sigma)_{\varepsilon^2}|_{g}+\frac{1}{\varepsilon^2}\int_{(\partial\Sigma)_{\varepsilon^2}}dA_{g}=|\Sigma\setminus (\partial\Sigma)_{\varepsilon^2}|_{g}+\frac{1}{\varepsilon^2}|(\partial\Sigma)_{\varepsilon^2}|_g.
\end{align}
In order to compute $|(\partial\Sigma)_{\varepsilon^2}|_g$, we use the normal coordinates $(r,\theta)$, centered at a point $p\in\partial\Sigma$. Then it is well-known that $g=dr^2+(r-(K(p)/6)r^3+O(r^4))^2d\theta^2$, where $K(p)$ is the Gauss curvature of $(\Sigma,g)$ at $p$. Then the area with respect to $g$ of the part of the disk centered at $p$ of radius $\varepsilon^2$ containing in $\Sigma$ is less than
$$
2\pi\int_0^{\varepsilon^2}\left(r-\frac{K(p)}{6}r^3+O(r^4)\right)\,dr=\pi\varepsilon^4+o(\varepsilon^8), \text{ as } \varepsilon\to 0.
$$
 Let $\partial_p\Sigma$ denote the connected component of $\partial \Sigma$, containing $p$. Hence, because of the overlapping, the $\varepsilon^2$-tubular neighborhood of $\partial_p\Sigma$ has area with respect to $g$ less than
$$
\left(\pi\varepsilon^4+o(\varepsilon^8)\right)|\partial_p\Sigma|_g,\text{ as } \varepsilon\to 0.
$$
Hence, $|(\partial\Sigma)_{\varepsilon^2}|_g<\pi\varepsilon^4|\partial\Sigma|_g+o(\varepsilon^8)$, as $\varepsilon\to 0$, where $|\partial\Sigma|_g$ is the total length of the boundary with respect to $g$. Coming back to~\eqref{eq:area}, we get
$$
|\Sigma|_{g}<|\Sigma|_{g_{\varepsilon}}=|\Sigma\setminus (\partial\Sigma)_{\varepsilon^2}|_{g}+\pi\varepsilon^2|\partial\Sigma|_g+o(\varepsilon^6), \text{ as } \varepsilon\to 0
$$
and we obtain the desired by the squeezing theorem. 

Further, $\lambda_i(g, \dfrac{1}{\varepsilon}\cot(r))\to-\infty$, as $\varepsilon\to 0$ by~\cite[Proposition 2.6]{hassannezhad2021nodal}. Hence, coming back to~\eqref{ineq:Xi+}, we get
$$
\Xi_{r,i}^+(\Sigma,g_{\varepsilon}) \to -\infty,~\text{ as }~\varepsilon\to 0.
$$
\end{proof}
  
We now investigate the critical points for these two functionals. We start with properties of Laplacian eigenvalues with Robin boundary conditions. Following the approach in \cite{lima2023eigenvalue}, we can prove the following lemma

\begin{lemma}
\label{lemma:alt derivative}
Let $\Sigma^k$ be a compact manifold with piecewise-smooth boundary. Suppose $(g_t)_t$ is a family of Riemannian metrics analytically indexed by $t\in(-\epsilon,\epsilon)$ with $g_0=g$. Denote $g_t^b$ the induced metric on $\partial\Sigma$ with $g^b=g_0^b$. For each prescribed  $\sigma\in\mathbb{R}$, let $\lambda_0^t\leqslant \lambda_1^t\leqslant \ldots$ be the Laplacian eigenvalues such that
\begin{equation}
\begin{cases}
\label{eqn:dirichlet to neumann eqn 1}
\Delta_{g_t} u=\lambda_i^t u\ &\mathrm{in}\ \Sigma,  \\
\dfrac{\partial u}{\partial \eta^t}= \sigma u\ &\mathrm{on}\ \partial\Sigma.
\end{cases}
\end{equation}
Denote $E_i(g_t,\sigma)$ the eigenspace of $\lambda_i^t$ satisfying \eqref{eqn:dirichlet to neumann eqn 1}. Then the following hold.
\begin{enumerate}[(i)]
\item \label{lemma:item 1} Each eigenvalue $\lambda_i^t$ is Lipschitz in $t$, and admits a weak derivative in $t$.
\item \label{lemma:item 5} Suppose a family of metrics $(g_t)$ has derivatives $\dfrac{d}{dt}(g_t)=h_t$ on $\Sigma$ and $\dfrac{d}{dt}(g^b_t)=h_t^b$ on $\partial\Sigma$. Denote $dA_t$ and $dL_t$ the measures, induced by $g_t$ on $\Sigma$ and $\partial\Sigma$, respectively. We can find an $L^2(\Sigma,dA_t)$ normalized functions $\phi_i^t$ such that the weak derivative of $\lambda_i^t$ is given by the following formula
\begin{equation}
\begin{aligned}
\label{eqn: alt deri quadratic}
  -\int_{\Sigma} \left\langle d\phi_i^t\otimes d\phi_i^t+\dfrac{1}{4} \Delta_{g_t} (\phi_i^t)^2 g_t,   h_t \right\rangle_t dA_t-\dfrac{\sigma}{2}\int_{\partial\Sigma}(\phi_i^t)^2\langle h_t^b,g_t^b  \rangle_t dL_t.
\end{aligned}
\end{equation}
\end{enumerate}
\end{lemma}
From the lemma above, we see that for a smooth family $(g_t)_t$ of metrics, the functions $\Xi^+_{r,k}(g_t)$ and $\Xi^+_{r,k}(g_t)$ are Lipschitz, hence differentiable almost everywhere. To study the critical metrics of the functionals $\Xi^+_{r,k}$ and $\Xi^-_{r,k}$, we need the following statement on the weak derivatives of the normalized Laplace eigenvalues and of these functionals. 
\begin{lemma}
\label{lemma:normalized derivative}
For a smooth family $(g_t)_{t\in(-\varepsilon,\varepsilon)}$ of Riemannian metrics, let $h_t=\dfrac{d}{dt} (g_t)$. The weak derivative of the following normalized Laplace eigenvalue 
\begin{align}
\label{eqn:normalized eigenvalue}
\lambda_i\left(g_t, \sigma \right)\vert\Sigma\vert_{g_t}^{2/k},\ i\geqslant 0
\end{align}
is given by
\begin{equation}
\begin{aligned} 
\label{eqn:normalized eigenvalue derivative}
 Q_{t,i}(\sigma)(u^t)= & -A_t^{(2-k)/k}  \biggr( \int_{\Sigma}\left\langle A_t\left( du^t\otimes du^t+\dfrac{1}{4}\Delta_{g_t}(u^t)^2g _t\right)-\dfrac{\lambda_i^t}{k}g_t,h_t\right\rangle_t d\mu_t\\
 & +\dfrac{\sigma A_t}{2}\int_{\partial\Sigma} (u_t)^2 \langle g^b_t,h^b_t\rangle_t d\nu_t \biggr)
\end{aligned}
\end{equation}
for some $L^2(\Sigma,g_t)$-normalized $u^t\in E_i(g_t,\sigma)$, and  $A_t=\vert \Sigma\vert_{g_t}$. Given any fixed $i\geqslant 1$, the functions $\Xi^+_{r,i}(g_t)$ and $\Xi^-_{r,i}(g_t)$ are differentiable almost everywhere, and their weak derivative formulae can directly obtain from \eqref{eqn:normalized eigenvalue derivative} %
\end{lemma}

\begin{proof}
Recall that 
\begin{align*}
\dfrac{d}{dt}\left(\vert \Sigma\vert_{g_t} \right)=\dfrac{1}{2}\int_{\Sigma}\langle g_t,h_t\rangle_t d\mu_t.
\end{align*}
 Combining this with the formula \eqref{eqn: alt deri quadratic}, a direct computation gives the formula of the weak derivatives of the normalized Laplace eigenvalues \eqref{eqn:normalized eigenvalue derivative}.

For $i\geqslant 1$ given, the two normalized Laplace eigenvalues in the formula of $\Xi^+_{r,i}(g_t)$ (also of $\Xi^-_{r,i}$) are both differentiable almost everywhere. Fix $t_0$ such that the two normalized eigenvalues are both differentiable. The derivative of $\Xi^+_{r,i}(g_t)$ at $t=t_0$ is then given by the derivative of 
$Q_{t_0,0}(-\tan(r))(u^{t_0}_0)$ or $Q_{t,i}(\cot(r))(u_1^{t_0})$, depending on whether $\lambda_0(g_{t_0},-\tan r)$ or $\lambda_i(g_{t_0},\cot r)$ is smaller, respectively. For the case these two eigenvalues are equal, we have
\begin{align*}
\dfrac{d}{dt}\biggr\vert_{t=t_0}\Xi^+_{r,i}(g_t)=\min\lbrace Q_{t_0,0}(-\tan(r))(u^{t_0}_0), Q_{t,i}(\cot(r))(u_1^{t_0})\rbrace,
\end{align*}
where $u_0^{t_0}\in E_0(g_{t_0},-\tan r)$ and $u_1^{t_0}=E_i(g_{t_0}, \cot r)$.
A similar statement also holds for the functional $\Xi^-_{r,i}(g_t)$.
\end{proof}
With the result above, we shall be able to complete the proof of Theorem \ref{4.4}.

\begin{proof}[Proof of Theorem \ref{4.4}]
We first prove the first half of this theorem. For $u\in E_i(g,\sigma)$, define the mapping
 \begin{align*}
 T(u,\sigma):=\lbrace(\tau^k(u),F(u,\sigma))\in S^2(\Sigma)\times S^2(\partial\Sigma)  \rbrace,
 \end{align*}
 where 
 \begin{align*}
 &\tau^k(u):= A\left(du\otimes du +\dfrac{1}{4}\Delta_g(u^2)g\right)- \dfrac{\lambda_i}{k}g,\\
 & F(u,\sigma):=\dfrac{A\sigma}{2}u^2 g^b.
 \end{align*}
 Here $g$ is a critical metric of $\Xi^+_{r,i}$ by assumption. We use Lemma \ref{lemma:normalized derivative} and argue as in the Robin eigenvalue case in Section \ref{sec:2.1}. We can show for any  $$(h_1,h_2)\in L^2(S^2(\Sigma),g)\times L^2(S^2(\partial\Sigma),g), $$ there exist $L^2(\Sigma,g)$ normalized function $u,v\in E_0(g,-\tan r)\cup E_i(g,\cot r)$ such that
\begin{align}
\label{eqn:side derivative diff sign}
\langle T(u),(h_1,h_2)\rangle \langle T(v),(h_1,h_2)\rangle \leqslant 0.
\end{align}

Let $\widehat{C}$ be the convex hull generated by $T(u_0,-\tan r)$ and the following set
\begin{align*}
K_i=\lbrace T(u,\cot r):u\in E_i(g,\cot r), \Vert u\Vert_{L^2(g)}=1 \rbrace.
\end{align*}
Note the convex hull $\widehat{C}$ is compact.
 
Then we must have $(0,0)\in \widehat{C}$. Otherwise, as $\hat{C}$ is compact, there exists some $$(h_1,h_2)\in L^2(S^2(\Sigma),g)\times L^2(S^2(\partial\Sigma),g) $$ such that
\begin{align*}
\inf_{(\tau,F)\in\widehat{C}}\left\langle (\tau,F),(h_1,h_2)\right\rangle>\epsilon>0.
\end{align*}
This and \eqref{eqn:side derivative diff sign} lead to a contradiction. 

When $g$ is a critical metric of $\Xi^-_{r,i}$, an analogous argument shows the convex hull generated by $T(u_0,\tanh r)$ and the set
\begin{align*}
\lbrace T(u,\coth r): u\in E_i(g,\coth r), \Vert u\Vert_{L^2(g)}=1   \rbrace
\end{align*}
also contains the point $(0,0)$.

Then there exists non-negative constants $\lbrace t_{\alpha}\rbrace_{\alpha=0}^m$ with $\displaystyle\sum_{\alpha=0}^m t_{\alpha}=1$ and $L^2$-normalized $u_0\in E_0(g,\sigma_0)$ and   $\lbrace u_j \rbrace_{j=1}^m$ in $E_i(g,\sigma_1)$ such that
\begin{equation}
\label{eqn:alt sphere functional zero}
\left\{ \begin{aligned} 
& \sum_{\alpha=0}^m t_{\alpha}\tau^k(u_{\alpha})=0\ \mathrm{on}\ \Sigma,\\
& \sum_{\alpha=0}^m t_{\alpha} F(u_{\alpha})=0 \ \mathrm{on}\ \partial\Sigma,
\end{aligned} \right.
\end{equation} 
where $\sigma_0=-\tan r$ and $\sigma_1=\cot r$ for $\Xi^+_{r,i}$, and $\sigma_0=\tanh r$ and $\sigma_1=\coth r$ for $\Xi^-_{r,i}$.
The second line of \eqref{eqn:alt sphere functional zero} never holds for $\Xi^-_{r,i}$, since $\coth r$ and $\tanh r$ are both strictly positive when $r>0$. Thus, \eqref{eqn:alt hyperbolic functional} has no critical point. The functional \eqref{eqn:alt sphere functional} can only have critical points when the equality \eqref{eqn:laplacian eigenvalue equal} holds. If this is not the case, the derivative of  $\Xi^+_{r,i}(\Sigma,g_t)$ near $t=0$ is induced by $\lambda_0(g_t,-\tan r)$ if $$\lambda_0(g_0,-\tan r)<\lambda_i(g_0,\cot r)$$ and by $\lambda_i(g_t,\cot r)$ if $$\lambda_0(g_0,-\tan r)>\lambda_i(g_0,\cot r).$$
Note $\lambda_0$ is always simple and always differentiable. By a similar argument on the convex hull generated by $K_i$ instead of $\hat{C}$, one would get $(0,0)\in K_i$ which is clearly impossible.

Denote this common Laplace eigenvalue by $\lambda$ from now.  The divergence theorem applied with $\sigma=-\tan r<0$ shows $\lambda>0$. From above, we can assume now $0<t_{\alpha}<1$ for all $0\leqslant \alpha\leqslant m$. Define $v_{\alpha}=\sqrt{t_{\alpha}}u_{\alpha}$. The first line of \eqref{eqn:alt sphere functional zero} yields
\begin{align}
\label{eqn:alt sphere interior}
A \sum_{\alpha=0}^m \left[   dv_{\alpha}\otimes dv_{\alpha}-\dfrac{1}{2}(\vert \nabla v_{\alpha}\vert_g^2-\lambda v_{\alpha}^2 )g  \right]-\dfrac{\lambda}{k}g=0.
\end{align}
Taking trace of the equation above, we obtain
\begin{align}
\label{eqn:alt sphere interior trace}
A\sum_{\alpha=0}^m\vert \nabla v_{\alpha}\vert_g^2\left(1-\dfrac{k}{2} \right)=\dfrac{-k\lambda}{2}\left(A\sum_{\alpha=0}^m v_{\alpha}^2-\dfrac{2}{k}\right)
\end{align}
Let us first assume $k>2$, using the same trick as before, define 
\begin{align*}
f=\dfrac{4\lambda}{k-2}\sum_{\alpha=0}^m v_{\alpha}^2-\dfrac{4\lambda}{A(k-2)}.
\end{align*}
Then $f$ is a Laplace eigenfunction with negative eigenvalue $\dfrac{4\lambda}{2-k}$. Moreover, the second line of \eqref{eqn:alt sphere functional zero} gives
\begin{align}
\label{eqn:alt sphere boundary}
(-\tan (r)) v_0^2 +\cot r\sum_{j=0}^m v_j^2\equiv 0,\ \mathrm{on}\ \partial\Sigma.
\end{align}
Thus we have $\dfrac{\partial f}{\partial\eta}\equiv 0$ on $\partial\Sigma$. Hence on $\Sigma$, we see
\begin{align}
\label{eqn:alt sphere embedding}
\sum_{\alpha} v_{\alpha}^2-\dfrac{1}{A}\equiv 0.
\end{align}
For the special case $k=2$, \eqref{eqn:alt sphere interior trace} immediately gives \eqref{eqn:alt sphere embedding}.\\

Combining \eqref{eqn:alt sphere embedding} with \eqref{eqn:alt sphere boundary}, we find
\begin{align}
\label{eqn:alt sphere boundary ratio}
v_0=\dfrac{1}{\sqrt{A}}\cos r,\ \mathrm{and}\ \sum_{j=0}^m v_j^2=\dfrac{1}{A}\sin^2 r\ \mathrm{on}\ \partial\Sigma.
\end{align}
Substitute \eqref{eqn:alt sphere embedding} and \eqref{eqn:alt sphere boundary ratio} back to \eqref{eqn:alt sphere interior} we get
\begin{align*}
1-\dfrac{\lambda}{k}g(\eta,\eta)=0
\end{align*}
Since $\eta$ is a unit norm vector field for $g$, we necessarily obtain $\lambda=k$.
This shows the map $\sqrt{A}(v_0,v_1,\ldots,v_m)$ is an isometric immersion from $(\Sigma,g)$ into $\mathbb{S}^m_+$. It follows analogously as the previous case that the immersion is actually an FBMI into $\mathbb B^m(r)$.
\end{proof}

\bibliographystyle{plain} 
\bibliography{FBMIbib} 
\end{document}